\documentclass{amsart}
\usepackage{amsfonts}

\usepackage{hyperref}
\newtheorem{theorem}{Theorem}[section]
\newtheorem{lemma}[theorem]{Lemma}

\theoremstyle{definition}
\newtheorem{definition}[theorem]{Definition}

\theoremstyle{Corollary}

\theoremstyle{remark}
\newtheorem{remark}[theorem]{Remark}

\numberwithin{equation}{section}

\begin{document}

\title{ HARMONIC MAPS FROM $\mathbb{C}^{n}$ to K\"{a}hler manifolds}

\author{}
\address{}
\curraddr{} \email{}
\thanks{}

\author{Jianming Wan}
\address{Department of Mathematics, Northwest University, Xi'an 710127, China } \email{wanj\_m@aliyun.com }
\thanks{The research is supported by the National Natural Science Foundation of China
N0.11301416}



\date{}

\dedicatory{}


\begin{abstract}
In this paper, we shall prove that a harmonic map from $\mathbb{C}^{n}$ ($n\geq2$) to
any K\"{a}hler manifold must be holomorphic under an assumption of energy density. It can be considered as
a complex analogue of the Liouville type theorem for harmonic maps obtained by Sealey.
\end{abstract}

\maketitle

\section*{}

\specialsection*{}



\section{Introduction}
The classical Liouville theorem says that a bounded
harmonic function on $\mathbb{R}^{n}$ (or holomorphic function on $\mathbb{C}^{n}$) has to be constant. Sealey (see
\cite{[S]} or \cite{[X]}) gave an analogue for harmonic maps. He proved that
a harmonic map of finite energy from $\mathbb{R}^{n}$ ($n\geq2$) to
any Riemannian manifold must be a constant map. In this paper we
consider the complex analogue of Sealey's result, that is the following

\textbf{Question: } \emph{Must a harmonic map with finite
$\bar{\partial}$-energy from $\mathbb{C}^{n}$ ($n\geq2$) to any
K\"{a}hler manifold be
 holomorphic }?

On the other hand, from Siu-Yau's proof of Frankel conjecture \cite{[SY]}
(the key is to prove a stable harmonic map from $S^{2}$ to $\mathbb{C}\mathbb{P}^{n}$ is
 holomorphic or conjugate holomorphic), we know that
it is very important to study the holomorphicity of harmonic maps. So
the above question is obviously interesting. We hope that
it is true. But we do not know how to prove it. Our partial
result can be stated as follows

\begin{theorem}
Let $f$ be a harmonic map from
$\mathbb{C}^{n}$ ($n\geq2$) to any K\"{a}hler manifold. Let $e(f)$ be the energy
density  and $e^{''}(f)$ be the  $\bar{\partial}$-energy density. If
\begin{equation}
e(f)e^{''}(f)(p)=O(\frac{1}{R^{4n+\alpha}})\footnote{The notation $O$ means $e(f)e^{''}(f)(p)\leq C\frac{1}{R^{4n+\alpha}}$ for some $C>0$ and sufficiently large $R$.} \label{1.1}
\end{equation} for some $\alpha>0$, where $R$ denotes the distance from origin to $p$, then $f$ is a holomorphic map.
\end{theorem}

The condition \ref{1.1} implies that the $\bar{\partial}$-energy is finite. Since
$$(e^{''}(f))^{2}\leq e(f)e^{''}(f)=O(\frac{1}{R^{4n+\alpha}}),$$ one has $$e^{''}(f)=O(\frac{1}{R^{2n+\frac{\alpha}{2}}}).$$ This leads to
$$\int_{\mathbb{C}^{n}}e^{''}(f)dv<\infty.$$

Note that we do not have any curvature assumption for the target manifold.

We should mention some other related holomorphicity of harmonic maps. For instance, Dong \cite{[D]} established many holomorphicity under the assumption of target manifolds with strongly semi-negative
curvature. In \cite{[X1]} Xin obtained some holomorphicity of harmonic maps from a complete Riemann surface into $\mathbb{C}\mathbb{P}^{n}$.

If the target manifold is $\mathbb{C}^{m}$ ( in this case
every component of the map is a harmonic function), then the answer of above
question is positive (see \cite{[W]}).

The main idea of the proof of theorem 1.1 is to consider a one
parameter family of maps and study the $\bar{\partial}$-energy
variation.

The rest of the paper is organized as follows: Section 2 contains
some basic materials of harmonic maps; In section 3 we study the
first variation of $\bar{\partial}$-energy ; Theorem 1.1 is proved in the last section.

\section{preliminaries }
The materials in this section may be found in the book of Xin
\cite{[X]}.

\subsection{Basic concepts of harmonic maps}
Let $f$ be a smooth map between two Riemannian manifolds ($M, g$)
and ($N, h$). We can define the energy density of $f$ by
$$e(f)=\frac{1}{2}trace|df|^{2}=\frac{1}{2}\sum_{i=1}^{m}\langle f_{*}e_{i},
f_{*}e_{i}\rangle,$$ where $\{e_{i}\}$ ($i=1, ..., m=dimM$) is a local
orthonormal frame field of $M$. The energy integral is defined by
$$E(f)=\int_{M}e(f)dv.$$ If we choose local coordinates
$\{x^{i}\}$ and $\{y^{\alpha}\}$ in $M$ and $N$, respectively, the energy density can be written
as
\begin{equation}
e(f)(x)=\frac{1}{2}g^{ij}(x)\frac{\partial
f^{\alpha}(x)}{\partial x^{i}}\frac{\partial f^{\beta}(x)}{\partial
x^{j}}h_{\alpha\beta}(f(x)). \label{f2.1}
\end{equation}
 The tension field of $f$ is
$$\tau(f)=(\nabla_{e_{i}}df)(e_{i}),$$ where $\nabla$ is the
induced connection on the pull-back bundle $f^{-1}TN$ over $M$ from those of $M$ and $N$.

\begin{definition}
We say that $f$ is a harmonic map if $\tau(f)=0$.
\end{definition}

From the variation point of view, a harmonic map can be seen as the
critical point of energy integral functional. Let $f_{t}$ be a one
parameter family of maps. We can regard it as a smooth map from
$M\times(-\epsilon, \epsilon)\rightarrow N$. Let $f_{0}=f$,
$\frac{df_{t}}{dt}|_{t=0}=v$. Then we have the first variation
formula (see \cite{[X]}) 
\begin{equation}\frac{d}{dt}E(f_{t})|_{t=0}=\int_{M}div
Wdv-\int_{M}\langle v, \tau(f)\rangle dv,  \label{f2.2}
\end{equation}
 where $W=<v,
f_{*}e_{j}>e_{j}$. If $M$ is compact, then $\int_{M}div Wdv=0$. We
know that a harmonic map is the critical point of energy
functional.
\subsection{$\bar{\partial}$-energy}
Let us consider the complex case.

Let $f$ be a smooth map from $\mathbb{C}^{n}$ to a K\"{a}hler
manifold $N$. Let $J$ be the standard complex structure of
$\mathbb{C}^{n}$ and $J^{'}$ be the complex structure of $N$.
$\omega$ and $\omega^{N}$ are the corresponding K\"{a}hler forms of
$\mathbb{C}^{n}$ and $N$ (i.e. $\omega(\cdot, \cdot)=\langle J\cdot,
\cdot\rangle$ and $\omega^{N}(\cdot, \cdot)=\langle J^{'}\cdot, \cdot\rangle$). The $\bar{\partial}$-energy density is defined by
\begin{eqnarray*}
e''(f)=|\bar{\partial} f|^{2} & = & |f_{*}J-J'f_{*}|^{2}\\
                            & = & \frac{1}{4}(|f_{*}e_{i}|^{2}+|f_{*}Je_{i}|^{2}-2\langle J^{'}f_{*}e_{i},f_{*}Je_{i}\rangle)\\
                            & = & \frac{1}{2}(e(f)-\langle f^{*}\omega^{N},\omega^{M}\rangle),\\
\end{eqnarray*}
where $\{e_{i}, Je_{i}\}$ ($i=1, ..., n$) is the Hermitian frame of
$\mathbb{C}^{n}$ and $\langle f^{*}\omega^{N},\omega\rangle$ denotes the induced
norm. We call that $f$ is holomorphic if $f_{*}J=J^{'}f_{*}$. Obviously $f$
is holomorphic if and only if $|\bar{\partial}f|^{2}\equiv0$.

It is well known that a holomorphic map between two K\"{a}hler manifolds must be harmonic (c.f. \cite{[X]}).

We denote $\bar{\partial}$-energy by
$$E_{\bar{\partial}}(f)=\int_{\mathbb{C}^{n}}|\bar{\partial}f|^{2}dv.$$

\section{$\bar{\partial}$-energy variation}
Let us consider the one parameter family of maps $f_{t}(x)=f(tx):
\mathbb{C}^{n}\longrightarrow N, t\in(1-\epsilon, 1+\epsilon)$ and
$f_{1}=f$. Let $B_{R}$ denote the Euclid ball in $\mathbb{C}^{n}$ of
radius $R$ around $0$. We write
$$E(R,t)=\int_{B_{R}}|\bar{\partial}f_{t}|^{2}dv.$$

\begin{lemma}
$E(R,t)=t^{2-2n}E(Rt,1)$. \label{l3.1}
\end{lemma}

\begin{proof}
Under the standard Hermitian metric of $\mathbb{C}^{n}$,
$g^{ij}=\delta_{ij}$, from \ref{f2.1} we have
$$e(f_{t})(x)=t^{2}e(f)(tx).$$ By using the natural coordinates, it is
 easy to show that
$$\langle f^{*}_{t}\omega^{N}, \omega\rangle(x)=t^{2}\langle f^{*}\omega^{N}, \omega\rangle(tx).$$
So we get
$$|\bar{\partial}f_{t}|^{2}(x)=t^{2}|\bar{\partial}f|^{2}(tx).$$
It is easy to check that
$$\int_{B_{R}}|\bar{\partial}f_{t}|^{2}dv=t^{2-2n}\int_{B_{Rt}}|\bar{\partial}f|^{2}dv.$$
Thus we obtain the lemma.
\end{proof}

We now prove the following variation formula of
$\bar{\partial}$-energy.
\begin{lemma}
$\frac{\partial E(R,t)}{\partial t}|_{t=1}=\frac{R}{2}\int_{\partial
B_{R}}(|f_{*}\frac{\partial}{\partial
r}|^{2}-\langle J^{'}f_{*}\frac{\partial}{\partial r},
f_{*}J\frac{\partial}{\partial r}\rangle)dv.$ \label{l3.2}
\end{lemma}
The proof will be separated in two steps.
\begin{proof}
Let $\{ e_{1},...,e_{2n}=\frac{\partial}{\partial r}\}$ be a local
orthonormal frame field, where $\frac{\partial}{\partial r}$ denotes
unit radial vector field. By the definition of $f_{t}(x)$, it is
easy to see that the variation vector field of $f_{t}$ at $t=1$ is
  $$v=\frac{df_{t}}{dt}|_{t=1}=rf_{*}\frac{\partial}{\partial r}.$$

\textbf{Step 1:} From \ref{f2.2} we have
\begin{eqnarray*}
\frac{d}{dt}\int_{B_{R}}e(f_{t})dv|_{t=1} & = & \int_{B_{R}}div\langle v,f_{*}e_{j}\rangle e_{j}dv-\int_{B_{R}}\langle v,\tau(f)\rangle dv\\
                                          & = & \int_{\partial B_{R}}\langle v,f_{*}\frac{\partial}{\partial r}\rangle dv\\
                                          & = & R\int_{\partial B_{R}}|f_{*}\frac{\partial}{\partial r}|^{2}dv.
\end{eqnarray*}
Since $f$ is harmonic, we know that the tension field $\tau(f)=0$ and
the second "$=$" follows from divergence theorem.

\textbf{ Step 2:} On the other hand, from \cite{[X]} we know that
$\frac{d}{dt}f^{*}_{t}\omega^{N}=d\theta_{t}$, where
$\theta_{t}=f^{*}_{t}i(f_{t*}\frac{\partial}{\partial
t})\omega^{N}$. Since
$\frac{df_{t}}{dt}|_{t=1}=rf_{*}\frac{\partial}{\partial r}$, we get
$\theta_{1}=\theta=rf^{*}i(f_{*}\frac{\partial}{\partial
r})\omega^{N}$. Then
\begin{eqnarray*}
&& \frac{d}{dt}\int_{B_{R}}\langle f^{*}_{t}\omega^{N}, \omega\rangle  dv|_{t=1}\\
 & = & \int_{B_{R}}\langle d\theta, \omega\rangle dv\\
 & = & \int_{B_{R}}d(\theta\wedge*\omega)+\int_{B_{R}}\langle\theta,\delta\omega\rangle dv\\
 & = & \int_{\partial B_{R}}\theta\wedge*\omega -\int_{B_{R}}\langle\theta, *d\omega^{n-1}\rangle dv\\
 & = & \int_{\partial B_{R}}\theta\wedge*\omega \\
 & = & -\int_{\partial B_{R}}\theta(e_{i})\omega(e_{i}, \frac{\partial}{\partial r})dv\\
 & = & -R\int_{\partial B_{R}}\omega^{N}(f_{*}\frac{\partial}{\partial r}, f_{*}e_{i})\omega(e_{i},\frac{\partial}{\partial r})dv\\
 & = & -R\int_{\partial B_{R}}\langle  J^{'}f_{*}\frac{\partial}{\partial r}, f_{*}e_{i}\rangle\langle Je_{i}, \frac{\partial}{\partial r}\rangle dv\\
 & = & R\int_{\partial B_{R}}\langle J^{'}f_{*}\frac{\partial}{\partial r}, f_{*}J\frac{\partial}{\partial r}\rangle dv.
\end{eqnarray*}
Note that $\langle d\theta, \omega\rangle dv=d\theta\wedge
*\omega$, the second "=" follows from the differential rules, where
$\delta$ and $*$ are the co-differential and star operators. By
stokes theorem and the definition of $\delta$, the third "=" holds.
The fifth "=" follows from direct computation. Since we may choose
$e_{1}=J\frac{\partial}{\partial r}$, the last "=" holds.

Combining step 1 and step 2, we obtain
$$\frac{d}{dt}\int_{B_{R}}|\bar{\partial}f_{t}|^{2}dv|_{t=1}=\frac{R}{2}\int_{\partial
B_{R}}(|f_{*}\frac{\partial}{\partial
r}|^{2}-\langle J^{'}f_{*}\frac{\partial}{\partial r},
f_{*}J\frac{\partial}{\partial r}\rangle)dv.$$ This completes the proof of
the lemma.
\end{proof}

\begin{remark}
If $M$ is a compact manifold , $\int_{M}\langle f^{*}\omega^{N},
\omega^{M}\rangle dv$ is a homotopy invariant. This was observed firstly by
Lichnerowicz \cite{[L]}.
\end{remark}

\section{Proof of theorem 1.1}
We use the similar trick in \cite{[S]}.

By lemma \ref{l3.1}, we obtain
$$\frac{\partial E(R,t)}{\partial t}|_{t=1}=(2-2n)E(R,1)+R\frac{\partial E(R,1)}{\partial
R}.$$

On the other hand, from lemma \ref{l3.2} and the condition \ref{1.1} in theorem 1.1, one has

\begin{eqnarray*}
\frac{\partial E(R,t)}{\partial t}|_{t=1}& = &\frac{R}{2}\int_{\partial
B_{R}}(|f_{*}\frac{\partial}{\partial
r}|^{2}-\langle J^{'}f_{*}\frac{\partial}{\partial r},
f_{*}J\frac{\partial}{\partial r}\rangle)dv\\
  &\geq&\frac{R}{2}\int_{\partial
B_{R}}(|f_{*}\frac{\partial}{\partial
r}|^{2}-|f_{*}\frac{\partial}{\partial
r}|\cdot|f_{*}J\frac{\partial}{\partial
r}|)dv\\
  &\geq&-\frac{R}{2}\int_{\partial
B_{R}}|f_{*}\frac{\partial}{\partial
r}|\cdot||f_{*}\frac{\partial}{\partial
r}|-|f_{*}J\frac{\partial}{\partial
r}||dv\\
 &\geq&-\frac{R}{2}\cdot R^{2n-1}\cdot \frac{1}{R^{2n+\frac{\alpha}{2}}}\cdot C\\
 & = &-\frac{C}{2}R^{-\frac{\alpha}{2}},
 \end{eqnarray*}
where $C$ is a positive constant. Hence for any
$\epsilon>0$, there exists an $R_{0}$ such that $$\frac{\partial E(R,t)}{\partial t}|_{t=1}\geq-\epsilon$$
 for all $R\geq R_{0}$. Therefore
$$R\frac{\partial E(R,1)}{\partial R}\geq -\epsilon+(2n-2)E(R,1)$$
for $R\geq R_{0}$.

If $E(\infty,1)=\int_{\mathbb{C}^{n}}|\bar{\partial}f|^{2}dv=E>0$,
then there exists a $R_{1}$ such that for all $R\geq R_{1}$, we have
$E(R,1)\geq E_{0}>0$. Since $n\geq2$ we can choose sufficiently small $\epsilon$ such that
$$R\frac{\partial E(R,1)}{\partial R}\geq A=-\epsilon+(2n-2)E_{0}>0,$$ when $R\geq
R_{2}=max(R_{0},R_{1})$. Then
$$E(\infty,1)=\int_{\mathbb{C}^{n}}|\bar{\partial}f|^{2}dv\geq
\int_{R_{2}}^{\infty}\frac{A}{R}dR=\infty.$$ It is a
contradiction. Therefore
$\int_{\mathbb{C}^{n}}|\bar{\partial}f|^{2}dv=0$. Hence $f$ is a
holomorphic map.

\begin{remark}
Compare with the real case \cite{[S]}, lemma \ref{l3.2} has the term $\langle J^{'}f_{*}\frac{\partial}{\partial r},
f_{*}J\frac{\partial}{\partial r}\rangle$. We need use condition \ref{1.1} to control it.
\end{remark}

\begin{remark}
 If we consider the $\partial$-energy density $e'(f)=|\partial f|^{2}=|f_{*}J+J'f_{*}|^{2}$, the corresponding result of theorem 1.1 also holds,
 i.e condition \ref{1.1} is replaced by $e(f)e^{'}(f)(p)=O(\frac{1}{R^{4n+\alpha}})$, the conclusion is that $f$ is a conjugate holomorphic map ($|\partial f|^{2}\equiv 0$).
\end{remark}

\bibliographystyle{amsplain}

\end{document}